\documentclass[12pt]{amsart}

\usepackage{amssymb,latexsym}

\usepackage{enumerate}

\usepackage{amsthm}
\usepackage{amscd}
\usepackage{amsmath}
\usepackage{delarray}
\usepackage{graphicx}
\usepackage{here}
\usepackage{float}

\makeatletter

\@namedef{subjclassname@2010}{%

  \textup{2010} Mathematics Subject Classification}

\makeatother

\newtheorem{thm}{Theorem}[section]

\newtheorem{cor}[thm]{Corollary}

\newtheorem{lem}[thm]{Lemma}

\theoremstyle{definition}

\newtheorem{rem}[thm]{Remark}

\numberwithin{equation}{section}

\newcommand{\N}{\mathbb{N}}
\newcommand{\Q}{\mathbb{Q}}
\newcommand{\R}{\mathbb{R}}
\newcommand{\I}{\mathbb{I}}
\newcommand{\FH}{\mathfrak{H}}

\newtheorem{prop}[thm]{Proposition}

\begin{document}

\title[rooted tree maps]{Rooted Tree Maps}

\author[T. Tanaka]{Tatsushi Tanaka}

\address{Department of Mathematics, Faculty of Science, Kyoto Sangyo University \endgraf
Motoyama, Kamigamo, Kita-Ku, Kyoto-City, 603-8555 Japan}

\email{t.tanaka@cc.kyoto-su.ac.jp}

\begin{abstract}

Based on Hopf algebra of rooted trees introduced by Connes and Kreimer, we construct a class of linear maps on noncommutative polynomial algebra in two indeterminates, namely rooted tree maps. 
We also prove that their maps induce a class of relations among multiple zeta values. 


\end{abstract}



\maketitle

\section{Introduction}\label{sec1}
A tree is a connected graph with no loops and a rooted tree is a tree with a special node called a root such that any edge is oriented away from it. 
We consider non-planar rooted trees which have no ordering of incoming edges for each vertex. 
Thanks to the non-planarity, we can define the free commutative algebra over $\Q$ generated by rooted trees. A product of rooted trees is sometimes called a rooted forest. 
%
%
An important operator on the algebra of rooted forests is the grafting operator $B_+$, which is a $\Q$-linear map defined by sending any rooted forest to a single tree by attaching the roots to a single new node which then becomes the new root. Because of non-planarity of rooted trees, there is a unique rooted forest $f$ for every rooted tree $t$ such that $t=B_+(f)$. 

It is known that the algebra $H$ of rooted trees is not only an algebra but a Hopf algebra (\cite{CK,Kre}). We also know that there exists the so-called Connes-Moscovici Hopf subalgebra $H_{\text{CM}}$ in $H$. 

Here comes a list of some notations in this paper. 
\begin{itemize}
\item $\Delta$ : the coproduct on $H$
\item $\FH :=\Q\langle x,y\rangle$, the noncommutative polynomial algebra over $\Q$ in $x$ and $y$
\item $\FH^1:=\Q+\FH y\supset \FH^0:=\Q+x\FH y$, the subalgebras of $\FH$
\item $M:\FH\otimes\FH\to\FH$ given by $M(v\otimes w)=vw$
\item $R_u$ : the right-concatenation map by $u$
\item $L_u$ : the left-concatenation map by $u$
\item $z:=x+y\in\FH$
\item $\Q[X]_{(d)}$ : the degree $d$ homogeneous part of the polynomial ring $\Q[X]$. 
\end{itemize}
Our first theorem is as follows. 
\begin{thm}\label{mthm1}
For any rooted forest $f(\neq\I)$, we 
can define the $\Q$-linear map from $\FH$ to $\FH$, which is also denoted by $f$, by 
\begin{itemize}
\item[(i)] If $f=\includegraphics[width=2.8mm,height=2.5mm]{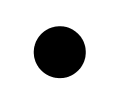}$, then $f(x):=xy$ and $f(y):=-xy$, 
\item[(i')] $B_+(f)(u):=R_yR_{y+z}R_y^{-1}f(u)$ for $u\in\{x,y\}$, 
\item[(i'')] If $f=gh$ with $g,h\neq\I$, then $f(u):=g(h(u))$ for $u\in\{x,y\}$,
\item[(ii)] For $w\in\FH$ and $u\in\{x,y\}$, $f(wu):=M(\Delta (f)(w\otimes u))$. 
\end{itemize}
\end{thm}

For the construction of rooted forest maps, it is convenient to introduce the additional map $\psi_f:=[f,R_x]$, where the bracket denotes the commutator. We call the number of nodes of a rooted forest $f$ the degree of $f$. Our second theorem states as follows. 
\begin{thm}\label{mthm2} 
For any rooted forests $f,g$, we show the following: 
\begin{itemize}
\item[(a)] There is a map $\phi_f$ such that $\psi_f=R_y\phi_f R_x$. 
\item[(b)] $f(\Q\cdot x+\Q\cdot y+\FH^0)\subset x \FH y$. 
\item[(c)] $\phi_{B_+(f)}=f+R_z\phi_f$. 
\item[(d)] $\phi_f\in\Q[R_z,\text{rooted tree maps}]_{(\deg f-1)}$. 
\item[(e)] $[f,g]=0$. 
\item[(f)] For any $v,w\in\FH$, $f(vw)=M(\Delta (f)(v\otimes w))$. 
\end{itemize}
\end{thm}

On the other hand, the multiple zeta values (abbreviated to MZV's) are defined, for an index $(k_1,\ldots ,k_r)\in\N^r$ with $k_1>1$, by the convergent series
\[
\zeta (k_1,\ldots ,k_r)=\sum_{m_1>\cdots >m_r>0}\frac{1}{{m_1}^{k_1}\cdots {m_r}^{k_r}}\quad \in\R . 
\]
It is known that there are many linear relatons among MZV's. 
For example, in \cite{T}, it is shown that the linear part of Kawashima relation \cite{Kaw} contains the quasi-derivation relation, which is a slightly but strictly larger class of relations than the derivation relation described in \cite{IKZ}. The quasi-derivation relation was first formulated in \cite{Kan} by modeling the Connes-Moscovici's Hopf algebra \cite{CM}. 

MZV's are often investigated under the algebraic language due to Hoffman \cite{H} which enables us to understand algebraic and combinatorial structures of MZV's in a down-to-earth way. 
The $\Q$-linear map $Z:\FH^0 \to \R$ called the evaluation map is defined by $Z(1)=1$ and
\[
Z(z_{k_1}\cdots z_{k_r})=\zeta(k_1,\ldots ,k_r) \quad (k_1>1), 
\]
where $z_k:=x^{k-1}y$ for $k\geq 1$. In what follows, all matters for MZV's are comprehended based on this algebraic setup. Here, note that to find a relation for MZV's amounts to find an element in $\ker Z$. 

As an application of rooted tree maps, we show the third theorem as follows. 
\begin{thm}\label{mthm3}
$f(\FH^0)\subset\ker Z$
for any rooted tree map $f$. 
\end{thm}
The proof is similar to the one we have discussed on the quasi-derivation relation in \cite{T}.



\vspace{5pt}

{\em Acknowledgements.} The author is grateful to scientific members and staffs in Max-Planck-Institut f\"{u}r Mathematik for their hospitality, where this work has been done. He is also thankful to Dr. Henrik Bachmann for helpful comments and advice. This work is also supported by Kyoto Sangyo University Research Grants.

\section{Rooted Trees}\label{sec2}

For the sake of conventions, we begin with a short review of the theory of rooted trees by Connes and Kreimer \cite{CK,Kre}.

\subsection{The algebra $H$ of rooted trees}\label{subsec2.1}
A tree is a non-empty connected finite graph with no loops and a rooted tree is a tree with a special node such that any edge is oriented away from it. The planarity of rooted trees is defined by taking a linear ordering of incoming edges for each vertex into account. In this paper we consider non-planar rooted trees and the topmost node represents the root. 

\begin{figure}
\centering
\includegraphics[width=6cm,height=1.2cm,clip]{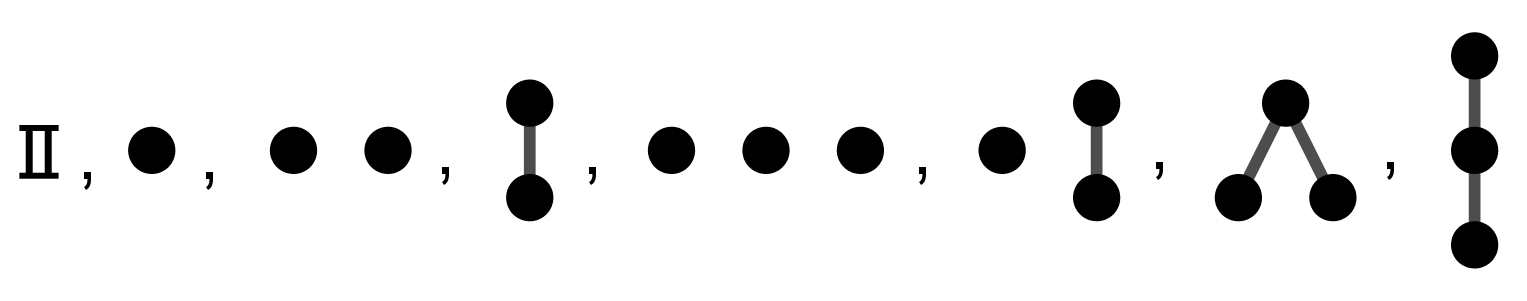}
\caption{Example of rooted forests}
\end{figure}

Let $H$ be the free commutative algebra over $\Q$ linearly generated by rooted forests: 
\[
H=\sum_{f : \text{rooted forest}}\Q\cdot f.
\]
Here the product of rooted trees is defined by the disjoint union. Thanks to the non-planarity, the product of trees is commutative. The neutral element is the empty forest denoted by $\I$ (this is {\em not} a tree but a forest). Obviously $H$ is algebraically generated by rooted trees. 
%

\subsection{Grafting operator}\label{subsec2.2}
Let $\mathcal{T}$ be the set of all rooted trees and $\langle\mathcal{T}\rangle_{\Q}$ be its linear span over $\Q$. The grafting operator is the $\Q$-linear map $B_+:H\to\langle\mathcal{T}\rangle_{\Q}$ defined by $B_+(\I)=\includegraphics[width=2.8mm,height=2.5mm]{deg1.png}$ and sending any rooted forest to a single tree by attaching the roots to a single new node which then becomes the new root: 
\[
B_+(t_1t_2\cdots t_n)=\raise -4mm \hbox{\includegraphics[width=2cm,height=1cm,clip]{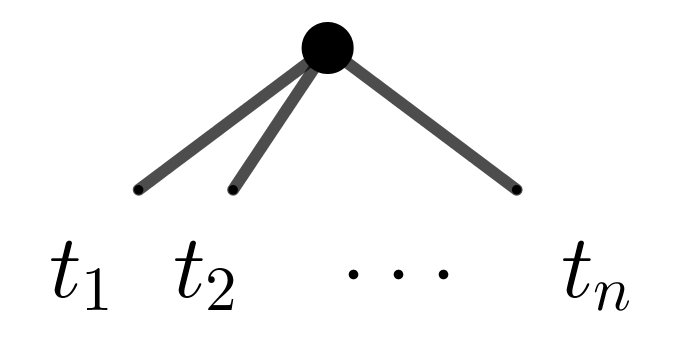}}
\]
for rooted trees $t_1,\ldots ,t_n$. 
%
%
Because of non-planarity of rooted trees, there is a unique forest $f$ for every rooted tree $t$ such that $t=B_+(f)$.

\subsection{Grading}\label{subsec2.3}
There is a natural grading on $H$ by the number of nodes. Let $\mathcal{F}_n$ be the set of all forests with $n$ nodes. Put $H_n:=\langle\mathcal{F}_n\rangle_{\Q}$ for $n\geq 1$ and $H_0:=\Q\I$. Then we have
\[
H=\bigoplus_{n\geq 0}H_n.
\]
The product has the grading property
\[
H_lH_k\subset H_{l+k}. 
\]

\subsection{Coproduct}\label{subsec2.4}
We define the coproduct $\Delta :H\to H\otimes H$. The coproduct is to be multiplicative, that is
\[
\Delta (fg)=\Delta (f)\Delta (g)
\]
and so we just need to define $\Delta (t)$ for tree $t$. 
Let $t=B_+(f)$, then we define $\Delta (t)$ by virtue of
\[
\Delta\circ B_+ =B_+\otimes\I + (id \otimes B_+)\circ\Delta, 
\]
that is
\begin{equation}\label{coprod}
\Delta (t)=\Delta\circ B_+(f):=t \otimes\I +(id \otimes B_+)\circ\Delta (f). 
\end{equation}
We also set $\Delta (\I)=\I\otimes\I$. 

This definition of $\Delta$ allows us to calculate the coproduct of rooted forests recursively. Here are some examples of coproducts of rooted trees and forests. 
\begin{align*}
\Delta(\includegraphics[width=2.8mm,height=2.5mm]{deg1.png})&=\Delta\circ B_+(\I) \\
&=B_+(\I)\otimes\I+(id \otimes B_+)\circ\Delta (\I) \\
&=\includegraphics[width=2.8mm,height=2.5mm]{deg1.png} \otimes\I+\I\otimes \includegraphics[width=2.8mm,height=2.5mm]{deg1.png} \\
\Delta (\includegraphics[width=6mm,height=2.3mm]{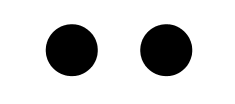})&=\includegraphics[width=6mm,height=2.3mm]{deg2-1.png} \otimes\I +2 \includegraphics[width=2.8mm,height=2.5mm]{deg1.png} \otimes \includegraphics[width=2.8mm,height=2.5mm]{deg1.png} +\I\otimes \includegraphics[width=6mm,height=2.3mm]{deg2-1.png} \\
\Delta (\raise -0.7mm \hbox{\includegraphics[width=3.6mm,height=4.2mm,clip]{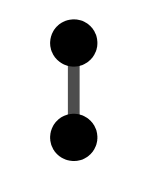}})&=\raise -0.7mm \hbox{\includegraphics[width=3.6mm,height=4.2mm,clip]{deg2-2.png}} \otimes\I+ \includegraphics[width=2.8mm,height=2.5mm]{deg1.png}\otimes \includegraphics[width=2.8mm,height=2.5mm]{deg1.png} +\I\otimes \raise -0.7mm \hbox{\includegraphics[width=3.6mm,height=4.2mm,clip]{deg2-2.png}} \\
\Delta (\raise -0.7mm \hbox{\includegraphics[width=4.8mm,height=4.2mm,clip]{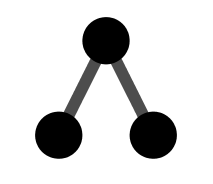}})&= \raise -0.7mm \hbox{\includegraphics[width=4.8mm,height=4.2mm,clip]{deg3-2.png}}\otimes\I +\includegraphics[width=6mm,height=2.3mm]{deg2-1.png}\otimes \includegraphics[width=2.8mm,height=2.5mm]{deg1.png} +2\includegraphics[width=2.8mm,height=2.5mm]{deg1.png} \otimes\raise -0.7mm \hbox{\includegraphics[width=3.6mm,height=4.2mm,clip]{deg2-2.png}} +\I\otimes \raise -0.7mm \hbox{\includegraphics[width=4.8mm,height=4.2mm,clip]{deg3-2.png}}
\end{align*}
\begin{prop}
The algebra morphism $\Delta$ is coassociative, that is 
\[
(id \otimes\Delta)\circ\Delta =(\Delta\otimes id )\circ\Delta .
\]
\end{prop}
\begin{proof}
The proof goes by induction on the grading. 
\end{proof}
\begin{rem}
(i) There is another geometric way to define the coproduct $\Delta$ by using admissible cuts, which can be found in \cite{CK, Kre}. \\ 
(ii) The counit $\hat{\I}:H\to\Q$ is given by vanishing on all forests except for $\hat{\I}(\I)=1$. The antipode $S:H\to H$ is defined by
\[
m\circ (S\otimes id )\circ\Delta =\I\circ\hat{\I}=m\circ (id \otimes S)\circ\Delta ,
\]
where $m$ denotes the product on $H$. For example, 
\[
S(\I)=\I,\quad S(\includegraphics[width=2.8mm,height=2.5mm]{deg1.png})=-\includegraphics[width=2.8mm,height=2.5mm]{deg1.png},\quad S(\raise -0.7mm \hbox{\includegraphics[width=3.6mm,height=4.2mm,clip]{deg2-2.png}})=-\raise -0.7mm \hbox{\includegraphics[width=3.6mm,height=4.2mm,clip]{deg2-2.png}}+\includegraphics[width=6mm,height=2.3mm]{deg2-1.png},\quad S(\includegraphics[width=6mm,height=2.3mm]{deg2-1.png})=\includegraphics[width=6mm,height=2.3mm]{deg2-1.png}.
\]
Then it is known that $(H,m,\I,\Delta ,\hat{\I},S)$ forms a Hopf algebra (Hopf algebra of rooted trees). 
\end{rem}

\subsection{Natural growth}\label{subsec2.5}
Let $N:H\to H$ be the $\Q$-linear map defined by
\[
N(\I):= \includegraphics[width=2.8mm,height=2.5mm]{deg1.png},\quad N(t):=\sum_{v\text{: node of }t}t_v\ (t\neq\I)
,\]
where $t_v:=t\,\circ_v \includegraphics[width=2.8mm,height=2.5mm]{deg1.png}$ which stands for grafting a single leaf to the vertex $v$ of $t$. We also require it to be a derivation on the augmentation ideal $\bigoplus_{n\geq 1}H_n$. 
\begin{figure}[H]
\centering
\includegraphics[width=4cm,height=1.2cm,clip]{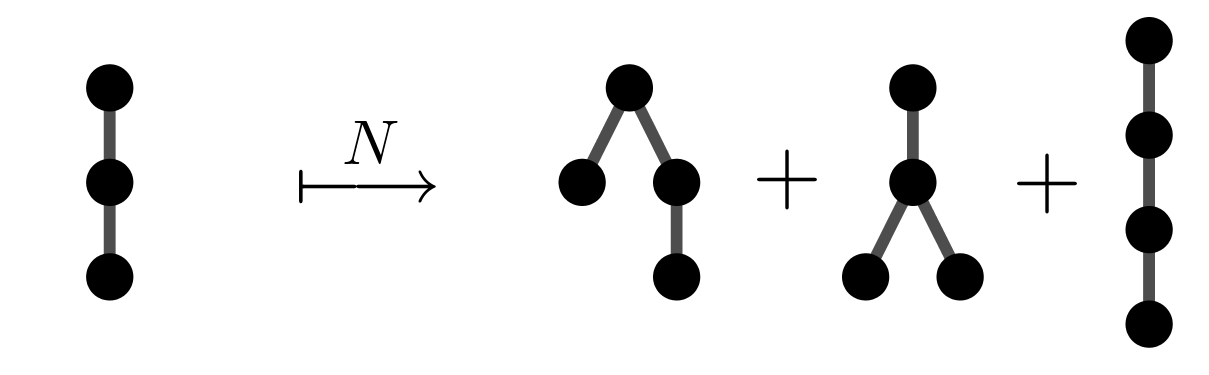}\qquad \raise 1mm \hbox{\includegraphics[width=4.5cm,height=1cm,clip]{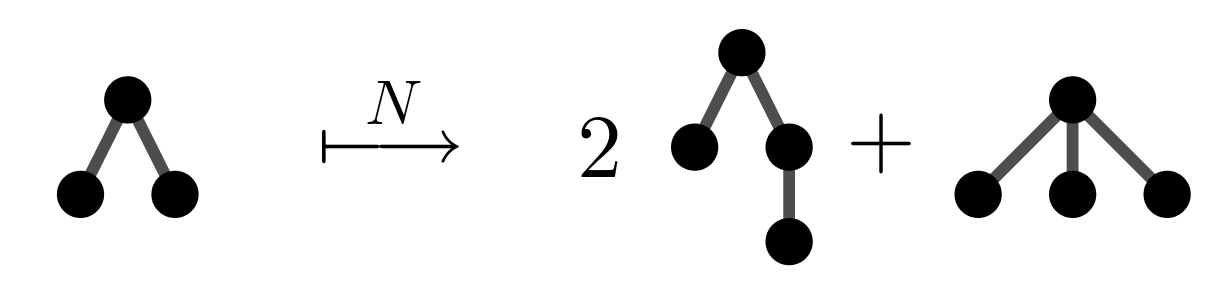}}
\caption{Example of natural growth}
\end{figure}
\noindent Let $\delta_k:=N^k(\I)$ for $k\geq 0$. For example, 
\[
\delta_1= \includegraphics[width=2.8mm,height=2.5mm]{deg1.png},\quad \delta_2= \raise -0.7mm \hbox{\includegraphics[width=3.6mm,height=4.2mm,clip]{deg2-2.png}},\quad \delta_3=\raise -1.8mm \hbox{\includegraphics[width=4.6mm,height=6mm,clip]{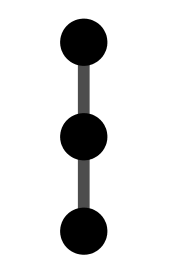}} +\raise -0.7mm \hbox{\includegraphics[width=4.8mm,height=4.2mm,clip]{deg3-2.png}},\quad \delta_4=3\,\raise -1.8mm \hbox{\includegraphics[width=4.8mm,height=6mm]{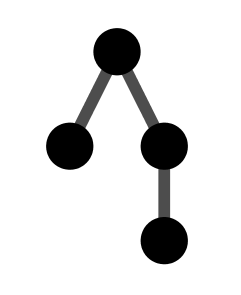}}+\raise -0.7mm \hbox{\includegraphics[width=6.6mm,height=4.2mm]{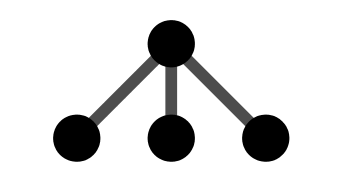}}+\raise -1.8mm \hbox{\includegraphics[width=4.8mm,height=6mm]{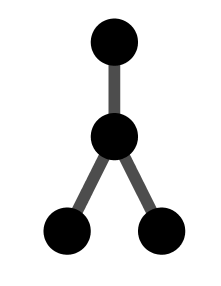}}+\raise -2.8mm \hbox{\includegraphics[width=3.6mm,height=7.8mm]{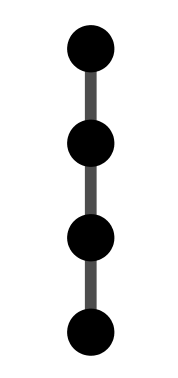}}. 
\]
\begin{prop}[\cite{CK}]
The elements $\delta_k\in H\ (k\geq 0)$ generate a Hopf subalgebra $H_{\text{CM}}\subset H$ which is called the Connes-Moscovici's Hopf subalgebra. 
\end{prop}


\section{Rooted Tree Maps}\label{sec3}


By Subsection \ref{subsec2.3} the space $H$ is graded by the degree.  In this section we construct rooted tree maps on $\FH$ inductively by this degree and show that they satisfy the proposition mentioned in the Introduction.

\subsection{Degree $0$ and $1$}\label{subsec3.1}
The only rooted forest of degree $0$ is $\I$, which is regarded as the identity map on $\FH$. We know that $\Delta (\I)=\I\otimes\I$ and $\I(vw)=\I(v)\I(w)$ for any $v,w\in\FH$. Put $\psi_{\I}=\phi_{\I}:=0$. It obviously follows that $\psi_{\I}=R_y\phi_{\I}R_x$. 

We see that $\mathcal{F}_1=\{\includegraphics[width=2.8mm,height=2.5mm]{deg1.png}\}$. The coproduct of \includegraphics[width=2.8mm,height=2.5mm]{deg1.png} is given by
\begin{equation}\label{eq19}
\Delta (\includegraphics[width=2.8mm,height=2.5mm]{deg1.png})= \includegraphics[width=2.8mm,height=2.5mm]{deg1.png}\otimes\I+\I\otimes\includegraphics[width=2.8mm,height=2.5mm]{deg1.png} 
\end{equation}
as stated in Section \ref{subsec2.4}. We define, for $w\in\FH$ and $u\in\{x,y\}$, the $\Q$-linear map \includegraphics[width=2.8mm,height=2.5mm]{deg1.png} by
\begin{equation}\label{eq18}
\includegraphics[width=2.8mm,height=2.5mm]{deg1.png}(wu)= \includegraphics[width=2.8mm,height=2.5mm]{deg1.png}(w)u+w \includegraphics[width=2.8mm,height=2.5mm]{deg1.png}(u) 
\end{equation}
and
\[
\includegraphics[width=2.8mm,height=2.5mm]{deg1.png}(x)=- \includegraphics[width=2.8mm,height=2.5mm]{deg1.png}(y)=xy.
\]
\begin{lem}
We have
\begin{equation}\label{eq1}
[\includegraphics[width=2.8mm,height=2.5mm]{deg1.png} ,R_z]=0. 
\end{equation}
\end{lem}
\begin{proof}
For $w\in\FH$, $[\includegraphics[width=2.8mm,height=2.5mm]{deg1.png} ,R_z](w)=\includegraphics[width=2.8mm,height=2.5mm]{deg1.png}(wz)-\includegraphics[width=2.8mm,height=2.5mm]{deg1.png}(w)z=\includegraphics[width=2.8mm,height=2.5mm]{deg1.png}(w)z+w\includegraphics[width=2.8mm,height=2.5mm]{deg1.png}(z)-\includegraphics[width=2.8mm,height=2.5mm]{deg1.png}(w)z=0. $
\end{proof}

Then we are allowed to define the map associated to \includegraphics[width=2.8mm,height=2.5mm]{deg1.png} by
\[
\psi_{\includegraphics[width=2mm,height=1.8mm]{deg1.png}}:=\text{sgn}(u)[\includegraphics[width=2.8mm,height=2.5mm]{deg1.png} ,R_u]=[\includegraphics[width=2.8mm,height=2.5mm]{deg1.png} ,R_x],
\]
where $\text{sgn}(u)=1$ or $-1$ according to $u=x$ or $y$. Since $\psi_{\includegraphics[width=2mm,height=1.8mm]{deg1.png}}(w)=w \includegraphics[width=2.8mm,height=2.5mm]{deg1.png}(x)=wxy$, it follows that
\begin{equation}\label{eq3}
\psi_{\includegraphics[width=2mm,height=1.8mm]{deg1.png}}=R_y\phi_{\includegraphics[width=2mm,height=1.8mm]{deg1.png}}R_x
\end{equation}
by putting $\phi_{\includegraphics[width=2mm,height=1.8mm]{deg1.png}}=id.$ This implies that
\begin{equation}\label{eq16}
 \includegraphics[width=2.8mm,height=2.5mm]{deg1.png}(wu)= \includegraphics[width=2.8mm,height=2.5mm]{deg1.png}(w) u+\text{sgn}(u)\phi_{\includegraphics[width=2mm,height=1.8mm]{deg1.png}}(wx)y\quad (w\in\FH,u\in\{x,y\})
\end{equation}
and in particular
\begin{equation}\label{eq17}
\includegraphics[width=2.8mm,height=2.5mm]{deg1.png}(\Q\cdot x+\FH^1)\subset\FH y
\end{equation}
because of $\includegraphics[width=2.8mm,height=2.5mm]{deg1.png}(1)=0$. We also find the following. 
\begin{lem}\label{lem3.2}
$\includegraphics[width=2.8mm,height=2.5mm]{deg1.png}(\Q\cdot x+\Q\cdot y+\FH^0)\subset x\FH y. $
\end{lem}
\begin{proof}
Using \eqref{eq16}, we have $\includegraphics[width=2.8mm,height=2.5mm]{deg1.png}(\Q\cdot x+\Q\cdot y+\FH^0)\subset x\FH$ by induction on the length of a word $w\in\FH^0$. We have already obtained \eqref{eq17}, and hence the lemma holds. 
\end{proof}

We obviously find that $[\includegraphics[width=2.8mm,height=2.5mm]{deg1.png} ,\includegraphics[width=2.8mm,height=2.5mm]{deg1.png} ]=0$. We also have the following. 
\begin{prop}
We have
\[
\includegraphics[width=2.8mm,height=2.5mm]{deg1.png}(vw)= \includegraphics[width=2.8mm,height=2.5mm]{deg1.png}(v)w+v \includegraphics[width=2.8mm,height=2.5mm]{deg1.png}(w)
\]
for any $v,w\in\FH$. 
\end{prop}
\begin{proof}
By \eqref{eq19} and \eqref{eq18}, we obtain the proposition by induction on the degree of a word $w$. 
\end{proof}

\subsection{Degree $2$}\label{subsec3.2}
In the first step, we prepare a lemma which is required several times below. 
\begin{lem}\label{lem9}
{\it If a $\mathbb{Q}$-linear map $f:\mathfrak{H}\to\mathfrak{H}$ satisfies $[f,R_x]=[f,R_y]=0$ and $f(1)=0$, Then $f\equiv 0$. }
\end{lem}
\begin{proof}
Since $f$ is $\mathbb{Q}$-linear, it is only necessary to show $f(w)=0$ for any words $w\in\mathfrak{H}$. Write $w=u_1u_2\cdots u_n$ with $u_1,u_2,\ldots ,u_n\in\{x,y\}$. Since $[f,R_{u_i}]=0$ for any $1\le i\le n$ by assumption, we have
$$ f(w)=f(u_1u_2\cdots u_n)=f(u_1u_2\cdots u_{n-1})u_n=\cdots =f(1)u_1u_2\cdots u_n=0. $$
\end{proof}

There are two rooted forests of degree $2$: \includegraphics[width=6mm,height=2.3mm]{deg2-1.png} and \raise -0.7mm \hbox{\includegraphics[width=3.6mm,height=4.2mm,clip]{deg2-2.png}}. Their coproducts are
\begin{equation}\label{eq2}
\Delta (\includegraphics[width=6mm,height=2.3mm]{deg2-1.png})=\includegraphics[width=6mm,height=2.3mm]{deg2-1.png} \otimes\I +2 \includegraphics[width=2.8mm,height=2.5mm]{deg1.png} \otimes \includegraphics[width=2.8mm,height=2.5mm]{deg1.png} +\I\otimes \includegraphics[width=6mm,height=2.3mm]{deg2-1.png}, \quad 
\Delta (\raise -0.7mm \hbox{\includegraphics[width=3.6mm,height=4.2mm,clip]{deg2-2.png}})=\raise -0.7mm \hbox{\includegraphics[width=3.6mm,height=4.2mm,clip]{deg2-2.png}} \otimes\I+ \includegraphics[width=2.8mm,height=2.5mm]{deg1.png}\otimes \includegraphics[width=2.8mm,height=2.5mm]{deg1.png} +\I\otimes \raise -0.7mm \hbox{\includegraphics[width=3.6mm,height=4.2mm,clip]{deg2-2.png}}. \end{equation}
We define, for $w\in\FH$ and $u\in\{x,y\}$, the $\Q$-linear maps \includegraphics[width=6mm,height=2.3mm]{deg2-1.png} and \raise -0.7mm \hbox{\includegraphics[width=3.6mm,height=4.2mm,clip]{deg2-2.png}} by
\begin{align}
\includegraphics[width=6mm,height=2.3mm]{deg2-1.png} (wu)&=\includegraphics[width=6mm,height=2.3mm]{deg2-1.png} (w)u+2\includegraphics[width=2.8mm,height=2.5mm]{deg1.png} (w)\includegraphics[width=2.8mm,height=2.5mm]{deg1.png} (u)+w\includegraphics[width=6mm,height=2.3mm]{deg2-1.png} (u), \label{eq20} \\
\raise -0.7mm \hbox{\includegraphics[width=3.6mm,height=4.2mm,clip]{deg2-2.png}} (wu)&=\raise -0.7mm \hbox{\includegraphics[width=3.6mm,height=4.2mm,clip]{deg2-2.png}} (w)u+\includegraphics[width=2.8mm,height=2.5mm]{deg1.png} (w)\includegraphics[width=2.8mm,height=2.5mm]{deg1.png} (u)+w\raise -0.7mm \hbox{\includegraphics[width=3.6mm,height=4.2mm,clip]{deg2-2.png}} (u) \label{eq21}
\end{align}
and for $u\in\{x,y\}$, 
\begin{equation}\label{eq7}
\includegraphics[width=6mm,height=2.3mm]{deg2-1.png} (u)=\includegraphics[width=2.8mm,height=2.5mm]{deg1.png} (\includegraphics[width=2.8mm,height=2.5mm]{deg1.png} (u)), \quad \raise -0.7mm \hbox{\includegraphics[width=3.6mm,height=4.2mm,clip]{deg2-2.png}} (u)=R_yR_{y+z}R_y^{-1}\includegraphics[width=2.8mm,height=2.5mm]{deg1.png} (u). 
\end{equation}
\begin{lem}\label{lem3.4}
We have 
\[
[\includegraphics[width=6mm,height=2.3mm]{deg2-1.png} ,R_z]=[\raise -0.7mm \hbox{\includegraphics[width=3.6mm,height=4.2mm,clip]{deg2-2.png}} ,R_z]=0.
\]
\end{lem}
\begin{proof}
For $w\in\FH$, we get by \eqref{eq20}
\[
[\includegraphics[width=6mm,height=2.3mm]{deg2-1.png} ,R_z](w)=\includegraphics[width=6mm,height=2.3mm]{deg2-1.png}(wz)-\includegraphics[width=6mm,height=2.3mm]{deg2-1.png}(w)z=\includegraphics[width=6mm,height=2.3mm]{deg2-1.png}(w)z+2\includegraphics[width=2.8mm,height=2.5mm]{deg1.png}(w)\includegraphics[width=2.8mm,height=2.5mm]{deg1.png}(z)
+w\includegraphics[width=6mm,height=2.3mm]{deg2-1.png}(z)-\includegraphics[width=6mm,height=2.3mm]{deg2-1.png}(w)z.
\]
Since $\includegraphics[width=6mm,height=2.3mm]{deg2-1.png}(z)=\includegraphics[width=2.8mm,height=2.5mm]{deg1.png}(z)=0$, this becomes $0$ and hence $[\includegraphics[width=6mm,height=2.3mm]{deg2-1.png} ,R_z]=0.$ The proof of $[\raise -0.7mm \hbox{\includegraphics[width=3.6mm,height=4.2mm,clip]{deg2-2.png}} ,R_z]=0$ goes similarly by using \eqref{eq21}. 
\end{proof}

Then we are allowed to define the maps associated to rooted forests of degree $2$ by
\[
\psi_{\includegraphics[width=4.5mm,height=1.8mm]{deg2-1.png}}:=\text{sgn}(u)[\includegraphics[width=6mm,height=2.3mm]{deg2-1.png} ,R_u]=[\includegraphics[width=6mm,height=2.3mm]{deg2-1.png} ,R_x],\quad \psi_{\raise -0.5mm \hbox{\includegraphics[width=2.7mm,height=3.1mm,clip]{deg2-2.png}}}:=\text{sgn}(u)[\raise -0.7mm \hbox{\includegraphics[width=3.6mm,height=4.2mm,clip]{deg2-2.png}} ,R_u]=[\raise -0.7mm \hbox{\includegraphics[width=3.6mm,height=4.2mm,clip]{deg2-2.png}} ,R_x].
\]
By the coproduct rules \eqref{eq2}, we calculate
\[
\psi_{\includegraphics[width=4.5mm,height=1.8mm]{deg2-1.png}}(w)=2\includegraphics[width=2.8mm,height=2.5mm]{deg1.png} (wx)y-wxzy,\quad \psi_{\raise -0.5mm \hbox{\includegraphics[width=2.7mm,height=3.1mm,clip]{deg2-2.png}}}(w)=\includegraphics[width=2.8mm,height=2.5mm]{deg1.png} (wx)y+wxzy, 
\]
and hence we have
\begin{equation}\label{eq4}
\psi_{\includegraphics[width=4.5mm,height=1.8mm]{deg2-1.png}}=R_y\phi_{\includegraphics[width=4.5mm,height=1.8mm]{deg2-1.png}}R_x,\quad \psi_{\raise -0.5mm \hbox{\includegraphics[width=2.7mm,height=3.1mm,clip]{deg2-2.png}}}=R_y\phi_{\raise -0.5mm \hbox{\includegraphics[width=2.7mm,height=3.1mm,clip]{deg2-2.png}}}R_x
\end{equation}
by putting
\[
\phi_{\includegraphics[width=4.5mm,height=1.8mm]{deg2-1.png}}=2\includegraphics[width=2.8mm,height=2.5mm]{deg1.png} -R_z,\quad \phi_{\raise -0.5mm \hbox{\includegraphics[width=2.7mm,height=3.1mm,clip]{deg2-2.png}}}=\includegraphics[width=2.8mm,height=2.5mm]{deg1.png} +R_z. 
\]
Notice that the property (c) in the Introduction holds for $f=\includegraphics[width=2.8mm,height=2.5mm]{deg1.png}$. These expressions and \eqref{eq1} implies that
\begin{equation}\label{eq10}
\phi_{\includegraphics[width=4.5mm,height=1.8mm]{deg2-1.png}},\phi_{\raise -0.5mm \hbox{\includegraphics[width=2.7mm,height=3.1mm,clip]{deg2-2.png}}}\in\Q[R_z,\includegraphics[width=2.8mm,height=2.5mm]{deg1.png}]_{(1)}
\end{equation}
by assuming the degree of $R_z$ to be $1$. Moreover, \eqref{eq4} implies that 
\begin{align}
\includegraphics[width=6mm,height=2.3mm]{deg2-1.png}(wu)=\includegraphics[width=6mm,height=2.3mm]{deg2-1.png}(w)u+\text{sgn}(u)\phi_{\includegraphics[width=4.5mm,height=1.8mm]{deg2-1.png}}(wx)y, \label{eq22} \\
\raise -0.7mm \hbox{\includegraphics[width=3.6mm,height=4.2mm,clip]{deg2-2.png}}(wu)=\raise -0.7mm \hbox{\includegraphics[width=3.6mm,height=4.2mm,clip]{deg2-2.png}}(w)u+\text{sgn}(u)\phi_{\raise -0.5mm \hbox{\includegraphics[width=2.7mm,height=3.1mm,clip]{deg2-2.png}}}(wx)y \label{eq23}
\end{align}
for $w\in\FH$ and $u\in\{x,y\}$, and in particular
\begin{equation}\label{eq24}
\includegraphics[width=6mm,height=2.3mm]{deg2-1.png}(\Q\cdot x+\FH^1), \raise -0.7mm \hbox{\includegraphics[width=3.6mm,height=4.2mm,clip]{deg2-2.png}}(\Q\cdot x+\FH^1)\subset\FH y 
\end{equation}
because of $\includegraphics[width=6mm,height=2.3mm]{deg2-1.png}(1)=\raise -0.7mm \hbox{\includegraphics[width=3.6mm,height=4.2mm,clip]{deg2-2.png}}(1)=0$. We also find the following. 
\begin{lem}\label{lem3.5}
$\includegraphics[width=6mm,height=2.3mm]{deg2-1.png}(\Q\cdot x+\Q\cdot y+\FH^0), \raise -0.7mm \hbox{\includegraphics[width=3.6mm,height=4.2mm,clip]{deg2-2.png}}(\Q\cdot x+\Q\cdot y+\FH^0)\subset x\FH y$.
\end{lem}
\begin{proof}
Using \eqref{eq22} and \eqref{eq23}, we have $f(\Q\cdot x+\Q\cdot y+\FH^0)\subset x\FH$, where $f=\includegraphics[width=6mm,height=2.3mm]{deg2-1.png}$ or \raise -0.7mm \hbox{\includegraphics[width=3.6mm,height=4.2mm,clip]{deg2-2.png}}, by induction on the length of a word $w\in\FH^0$. We have already obtained \eqref{eq24}, and hence the lemma holds. 
\end{proof}

Let $f$ be \includegraphics[width=6mm,height=2.3mm]{deg2-1.png} or \raise -0.7mm \hbox{\includegraphics[width=3.6mm,height=4.2mm,clip]{deg2-2.png}}. Because of
\[
[[f ,\includegraphics[width=2.8mm,height=2.5mm]{deg1.png} ],R_u]=-[[\includegraphics[width=2.8mm,height=2.5mm]{deg1.png} ,R_u],f]-[[R_u, f, \includegraphics[width=2.8mm,height=2.5mm]{deg1.png}],
\]
\eqref{eq3} and \eqref{eq4}, we see that
\begin{align}\label{eq5}
-\text{sgn}(u)[[f ,\includegraphics[width=2.8mm,height=2.5mm]{deg1.png} ],R_u]&=R_y\phi_{\includegraphics[width=2mm,height=1.8mm]{deg1.png}}[R_x, f]+R_y[\phi_{\includegraphics[width=2mm,height=1.8mm]{deg1.png}}, f]R_x+[R_y, f]\phi_{\includegraphics[width=2mm,height=1.8mm]{deg1.png}} R_x \\
&\quad -R_y\phi_f [R_x, \includegraphics[width=2.8mm,height=2.5mm]{deg1.png}]-R_y[\phi_f , \includegraphics[width=2.8mm,height=2.5mm]{deg1.png}]R_x-[R_y, \includegraphics[width=2.8mm,height=2.5mm]{deg1.png}]\phi_f R_x. \nonumber
\end{align}
But since we have already obtained $[\phi_{\includegraphics[width=2mm,height=1.8mm]{deg1.png}}, f]=0$ and $[\phi_f, \includegraphics[width=2.8mm,height=2.5mm]{deg1.png}]=0$ by \eqref{eq10}, we have
\begin{align*}
\eqref{eq5}&=-R_y\phi_{\includegraphics[width=2mm,height=1.8mm]{deg1.png}}\psi_f +\psi_f\phi_{\includegraphics[width=2mm,height=1.8mm]{deg1.png}}R_x+R_y\phi_f\psi_{\includegraphics[width=2mm,height=1.8mm]{deg1.png}}-\psi_{\includegraphics[width=2mm,height=1.8mm]{deg1.png}}\phi_f R_x \\
&=-R_y\phi_{\includegraphics[width=2mm,height=1.8mm]{deg1.png}}R_y\phi_f R_x+R_y\phi_f R_x\phi_{\includegraphics[width=2mm,height=1.8mm]{deg1.png}}R_x+R_y\phi_f R_y\phi_{\includegraphics[width=2mm,height=1.8mm]{deg1.png}}R_x-R_y \phi_{\includegraphics[width=2mm,height=1.8mm]{deg1.png}}R_x\phi_f R_x \\
&=-R_y\phi_{\includegraphics[width=2mm,height=1.8mm]{deg1.png}}R_z\phi_f R_x+R_y\phi_f R_z\phi_{\includegraphics[width=2mm,height=1.8mm]{deg1.png}}R_x. 
\end{align*}
This becomes $0$ since the maps $\phi_{\includegraphics[width=2mm,height=1.8mm]{deg1.png}}, \phi_f , R_z$ are commutative pairwise. By \eqref{eq7}, we see that $ f(1)=0$. Hence by Lemma \ref{lem9}, we have
\begin{equation}\label{eq8}
[\includegraphics[width=6mm,height=2.3mm]{deg2-1.png} , \includegraphics[width=2.8mm,height=2.5mm]{deg1.png}]=[\raise -0.7mm \hbox{\includegraphics[width=3.6mm,height=4.2mm,clip]{deg2-2.png}} , \includegraphics[width=2.8mm,height=2.5mm]{deg1.png}]=0. 
\end{equation}

Similarly, because of
\[
[[\includegraphics[width=6mm,height=2.3mm]{deg2-1.png} , \raise -0.7mm \hbox{\includegraphics[width=3.6mm,height=4.2mm,clip]{deg2-2.png}}],R_u]=-[[\raise -0.7mm \hbox{\includegraphics[width=3.6mm,height=4.2mm,clip]{deg2-2.png}} ,R_u], \includegraphics[width=6mm,height=2.3mm]{deg2-1.png}]-[[R_u, \includegraphics[width=6mm,height=2.3mm]{deg2-1.png}], \raise -0.7mm \hbox{\includegraphics[width=3.6mm,height=4.2mm,clip]{deg2-2.png}}]
\]
and \eqref{eq4}, we see that
\begin{align}\label{eq9}
-\text{sgn}(u)[[\includegraphics[width=6mm,height=2.3mm]{deg2-1.png} , \raise -0.7mm \hbox{\includegraphics[width=3.6mm,height=4.2mm,clip]{deg2-2.png}}],R_u]&=R_y\phi_{\raise -0.5mm \hbox{\includegraphics[width=2.7mm,height=3.1mm,clip]{deg2-2.png}}}[R_x, \includegraphics[width=6mm,height=2.3mm]{deg2-1.png}]+R_y[\phi_{\raise -0.5mm \hbox{\includegraphics[width=2.7mm,height=3.1mm,clip]{deg2-2.png}}}, \includegraphics[width=6mm,height=2.3mm]{deg2-1.png}]R_x+[R_y, \includegraphics[width=6mm,height=2.3mm]{deg2-1.png}]\phi_{\raise -0.5mm \hbox{\includegraphics[width=2.7mm,height=3.1mm,clip]{deg2-2.png}}}R_x \\
&\quad -R_y\phi_{\includegraphics[width=4.5mm,height=1.8mm]{deg2-1.png}}[R_x, \raise -0.7mm \hbox{\includegraphics[width=3.6mm,height=4.2mm,clip]{deg2-2.png}}]-R_y[\phi_{\includegraphics[width=4.5mm,height=1.8mm]{deg2-1.png}}, \raise -0.7mm \hbox{\includegraphics[width=3.6mm,height=4.2mm,clip]{deg2-2.png}}]R_x-[R_y, \raise -0.7mm \hbox{\includegraphics[width=3.6mm,height=4.2mm,clip]{deg2-2.png}}]\phi_{\includegraphics[width=4.5mm,height=1.8mm]{deg2-1.png}}R_x. \nonumber
\end{align}
By \eqref{eq10}, \eqref{eq8} and Lemma \ref{lem3.4}, we have
\[
[\phi_{\raise -0.5mm \hbox{\includegraphics[width=2.7mm,height=3.1mm,clip]{deg2-2.png}}}, \includegraphics[width=6mm,height=2.3mm]{deg2-1.png}]=[\phi_{\includegraphics[width=4.5mm,height=1.8mm]{deg2-1.png}}, \raise -0.7mm \hbox{\includegraphics[width=3.6mm,height=4.2mm,clip]{deg2-2.png}}]=0, 
\]
and hence
\begin{align*}
\eqref{eq9}&=-R_y\phi_{\raise -0.5mm \hbox{\includegraphics[width=2.7mm,height=3.1mm,clip]{deg2-2.png}}}\psi_{\includegraphics[width=4.5mm,height=1.8mm]{deg2-1.png}}+\psi_{\includegraphics[width=4.5mm,height=1.8mm]{deg2-1.png}}\phi_{\raise -0.5mm \hbox{\includegraphics[width=2.7mm,height=3.1mm,clip]{deg2-2.png}}}R_x+R_y\phi_{\includegraphics[width=4.5mm,height=1.8mm]{deg2-1.png}}\psi_{\raise -0.5mm \hbox{\includegraphics[width=2.7mm,height=3.1mm,clip]{deg2-2.png}}}-\psi_{\raise -0.5mm \hbox{\includegraphics[width=2.7mm,height=3.1mm,clip]{deg2-2.png}}}\phi_{\includegraphics[width=4.5mm,height=1.8mm]{deg2-1.png}}R_x \\
&=-R_y\phi_{\raise -0.5mm \hbox{\includegraphics[width=2.7mm,height=3.1mm,clip]{deg2-2.png}}}R_y\phi_{\includegraphics[width=4.5mm,height=1.8mm]{deg2-1.png}}R_x+R_y\phi_{\includegraphics[width=4.5mm,height=1.8mm]{deg2-1.png}}R_x\phi_{\raise -0.5mm \hbox{\includegraphics[width=2.7mm,height=3.1mm,clip]{deg2-2.png}}}R_x+R_y\phi_{\includegraphics[width=4.5mm,height=1.8mm]{deg2-1.png}}R_y\phi_{\raise -0.5mm \hbox{\includegraphics[width=2.7mm,height=3.1mm,clip]{deg2-2.png}}}R_x-R_y\phi_{\raise -0.5mm \hbox{\includegraphics[width=2.7mm,height=3.1mm,clip]{deg2-2.png}}}R_x\phi_{\includegraphics[width=4.5mm,height=1.8mm]{deg2-1.png}}R_x \\
&=-R_y\phi_{\raise -0.5mm \hbox{\includegraphics[width=2.7mm,height=3.1mm,clip]{deg2-2.png}}}R_z\phi_{\includegraphics[width=4.5mm,height=1.8mm]{deg2-1.png}}R_x+R_y\phi_{\includegraphics[width=4.5mm,height=1.8mm]{deg2-1.png}}R_z\phi_{\raise -0.5mm \hbox{\includegraphics[width=2.7mm,height=3.1mm,clip]{deg2-2.png}}}R_x, 
\end{align*}
which becomes $0$ since $\phi_{\includegraphics[width=4.5mm,height=1.8mm]{deg2-1.png}},\phi_{\raise -0.5mm \hbox{\includegraphics[width=2.7mm,height=3.1mm,clip]{deg2-2.png}}},R_z$ are commutative pairwise. Since $[\includegraphics[width=6mm,height=2.3mm]{deg2-1.png} , \raise -0.7mm \hbox{\includegraphics[width=3.6mm,height=4.2mm,clip]{deg2-2.png}}](1)=0$, we have
\[
[\includegraphics[width=6mm,height=2.3mm]{deg2-1.png} , \raise -0.7mm \hbox{\includegraphics[width=3.6mm,height=4.2mm,clip]{deg2-2.png}}]=0
\]
by Lemma \ref{lem9}. 
\begin{prop}
We have
\begin{align*}
\includegraphics[width=6mm,height=2.3mm]{deg2-1.png} (vw)&=\includegraphics[width=6mm,height=2.3mm]{deg2-1.png} (v)w+2\includegraphics[width=2.8mm,height=2.5mm]{deg1.png} (v)\includegraphics[width=2.8mm,height=2.5mm]{deg1.png} (w)+v\includegraphics[width=6mm,height=2.3mm]{deg2-1.png} (w), \\
\raise -0.7mm \hbox{\includegraphics[width=3.6mm,height=4.2mm,clip]{deg2-2.png}} (vw)&=\raise -0.7mm \hbox{\includegraphics[width=3.6mm,height=4.2mm,clip]{deg2-2.png}} (v)w+\includegraphics[width=2.8mm,height=2.5mm]{deg1.png} (v)\includegraphics[width=2.8mm,height=2.5mm]{deg1.png} (w)+v\raise -0.7mm \hbox{\includegraphics[width=3.6mm,height=4.2mm,clip]{deg2-2.png}} (w)
\end{align*}
for any $v,w\in\FH$. 
\end{prop}
\begin{proof}
By \eqref{eq2}, \eqref{eq20} and \eqref{eq21}, we obtain the proposition by induction on the degree of a word $w$. 
\end{proof}

\subsection{General degree}\label{subsec3.3}
Suppose that we have constructed the rooted tree (or forest) maps which degrees are less than $n$. Moreover we assume (a), (b), (d) and (e) in the Introduction for any rooted forest maps $f,g$ each of which degrees is less than $n$. We construct all of the rooted forest maps of degree $n$ and show that they satisfy (a), (b), (d) and (e). 

For any rooted forest $f$ with $\deg f=n>1$, $w\in\FH\backslash\Q$ and $u\in\{x,y\}$, we define
\begin{equation}\label{eq25}
f(wu):=M(\Delta (f)(w\otimes u)). 
\end{equation}
We also define, for $u\in\{x,y\}$, 
\[
f(u):=R_yR_{y+z}R_y^{-1}g(u)
\]
if $f$ is a tree and $f=B_+(g)$, or otherwise
\[
f(u):=g(h(u)),
\]
where $f=gh$ with non-empty rooted forests  $g$ and $h$. We notice that, in the case of $f=B_+(g)$, the definition of $f(u)$ makes sense because $g(\Q\cdot x+\Q\cdot y +\FH^0)\subset x\FH y$. 

By definition, it follows that $f(x)=-f(y)$, or equivalently $f(z)=0$. For any rooted forest of degree $<n$, we have obtained the same property. Hence we are allowed to define the map
\[
\psi_f:=\text{sgn}(u)[f,R_u]=[f,R_x]
\]
for $u\in\{x,y\}$. 

\vspace{10pt}

{\bf Case I:} \underline{$f$ is a tree, i.e. $f=B_+(g)$.}

By \eqref{coprod} and using Sweedler notation $\Delta (g)=\sum a\otimes b$, we find
\begin{align}\label{eq11}
\psi_f (w)&=\psi_{B_+(g)}(w)=B_+(g)(wx)-B_+(g)(w)x=M(((id \otimes B_+)\circ\Delta)(g)(w\otimes x)) \nonumber \\
&=g(w)xy+\sum_{b\neq\I}a(w)R_yR_{y+z}R_y^{-1}b(x).
\end{align}
Note that, for the last equality, we use $B_+(\I)=\includegraphics[width=2.8mm,height=2.5mm]{deg1.png}$ and $B_+(f)=R_yR_{y+z}R_y^{-1}f$ for $f\neq\I$. 
Since $a(w)x=R_x a(w)=(aR_x-\psi_a)(w)=(a-R_y\phi_a)(wx)$, $b(x)\in x\FH y$ and again $\Delta (g)=\sum a\otimes b$, 
\begin{align*}
\eqref{eq11}&=g(wx)y-\sum_{b\neq\I}a(w)b(x)y+\sum_{b\neq\I}(a-R_y\phi_a)(wx)L_x^{-1}R_yR_{y+z}R_y^{-1}b(x) \\
&=g(wx)y+\sum_{b\neq\I}(a-R_y\phi_a)(wx)L_x^{-1}(R_y^{-1}b(x))zy.
\end{align*}
Therefore we obtain $\psi_f=R_y\phi_f R_x$, where
\begin{equation}\label{eq12}
\phi_f =g+R_z\sum_{b\neq\I}R_{L_x^{-1}R_y^{-1}b(x)}(a-R_y\phi_a). 
\end{equation}

On the other hand, we find
\begin{align*}
\psi_g (w)&=\sum_{a\neq g}a(w)b(x)=\sum_{a\neq g}a(w)xL_x^{-1}b(x) \\
&=\sum_{a\neq g}(a-R_y\phi_a)(wx)L_x^{-1}b(x)
\end{align*}
and hence by putting
\begin{equation}\label{eq13}
\phi_g =\sum_{a\neq g}R_{L_x^{-1}R_y^{-1}b(x)}(a-R_y\phi_a), 
\end{equation}
we find $\psi_g=R_y\phi_g R_x$. 
Obviously, the condition $a\neq g$ is equivalent to the condition $b\neq\I$. Combining \eqref{eq12} and \eqref{eq13}, we have
\[
\phi_f =g+R_z\phi_g .
\]
(This is (c) in the Introduction.) This in particular asserts that $\phi_f\in\Q[R_z,\mathcal{T}_{n-1}]_{(n-1)}$, where $\mathcal{T}_{n-1}$ stands for the set of all rooted trees of degree $\leq n-1$. 

\vspace{15pt}

{\bf Case II:} \underline{$f$ is not a tree, i.e. $f=gh$ with $g,h\neq\I$.}

By easy calculation we find
\[
\psi_{gh}=g\psi_h+\psi_g h. 
\]
Since $\psi_g=R_y\phi_g R_x$ and $\psi_h=R_y\phi_h R_x$, we have
\begin{align*}
\psi_{gh}&=g R_y\phi_h R_x+R_y\phi_g R_x h \\
&=(R_y g-\psi_g)\phi_h R_x +R_y\phi_g (h R_x -\psi_h) \\
&=R_y (g\phi_h +\phi_g h -\phi_g R_z \phi_h)R_x. 
\end{align*}
Therefore we obtain $\psi_f=R_y\phi_f R_x$, where
\[
\phi_f =g\phi_h +\phi_g h -\phi_g R_z \phi_h .
\]
This in particular asserts again that $\phi_f\in\Q[R_z,\mathcal{T}_{n-1}]_{(n-1)}$. 

\vspace{15pt}

Therefore, for any rooted forest of degree $n$, we obtain (a) and (d) in the Introduction. For any rooted forest $f$ of degree $n$, we see that $f(1)=0$. Thereby we also have (b) in the Introduction by induction on a degree of a word in $\FH$. (The proof goes similar to Lemma \ref{lem3.2} and \ref{lem3.5}.)

Now the only we have to show is (e) in the Introduction for any rooted forests $f,g$ of degree $\leq n$. For rooted forests $f$ and $g$, we have
\[
[[f,g],R_u]=-[[g,R_u],f]-[[R_u,f],g]. 
\]
If $\deg f, \deg g\leq n$, because of this and (a), we see that
\begin{align}
-\text{sgn}(u)[[f,g],R_u]&=[\psi_f,g]-[\psi_g,f] \nonumber \\
&=R_y\phi_f[R_x,g]+R_y[\phi_f,g]R_x+[R_y,g]\phi_f R_x \nonumber \\
&\quad -R_y\phi_g[R_x,f]-R_y[\phi_g,f]R_x-[R_y,f]\phi_g R_x. \label{eq15}
\end{align}

If $f=\includegraphics[width=2.8mm,height=2.5mm]{deg1.png}$, then $\phi_f=id$ and hence
\[
[\phi_f,g]=0. 
\]
Since $\phi_g\in\Q[R_z,\mathcal{T}_{n-1}]_{(n-1)}$, we also have
\[
[\phi_g,f]=0. 
\]
Thus
\begin{align*}
\eqref{eq15}&=-R_y\psi_g+\psi_g R_x+R_y \phi_g\psi_f-\psi_f\phi_g R_x \\
&=-R_yR_z\phi_g R_x+R_y\phi_g R_z R_x=0. 
\end{align*}
Since $[f,g](1)=0$, we conclude $[f,g]=0$ by Lemma \ref{lem9}. 

Assume that $[f,g]=0$ holds for rooted forests $f,g$ of $\deg g=n$ and $\deg f\leq i$ with $1\leq i< n$. Then, for a rooted forest $f$ with $\deg f=i+1$, we have
\[
[\phi_f,g]=0,\quad [\phi_g,f]=0 
\]
because of (d): $\phi_f\in\Q [R_z,\mathcal{T}_i]_{(i)}, \phi_g\in\Q [R_z,\mathcal{T}_{n-1}]_{(n-1)}$. Thus
\begin{align*}
\eqref{eq15}&=-R_y\phi_f\psi_g+\psi_g\phi_f R_x+R_y \phi_g\psi_f-\psi_f\phi_g R_x \\
&=-R_y\phi_f R_z\phi_g R_x+R_y\phi_g R_z\phi_f R_x=0. 
\end{align*}
Since $[f,g](1)=0$, we conclude $[f,g]=0$ by Lemma \ref{lem9}. Thus we conclude (e), the commutativity property, for any rooted forests $f,g$ of degree $\leq n$. 
\begin{prop}
We have $f(vw)=M(\Delta (f)(v\otimes w))$ for any rooted forest map $f$ of defree $n$ and any $v,w\in\FH$. 
\end{prop}
\begin{proof}
By \eqref{eq25}, we obtain the proposition by induction on the degree of a word $w$. 
\end{proof}

As a consequence of this section, we have Theorem \ref{mthm1} and \ref{mthm2}.


\section{Application to MZV's}\label{sec4}

In this section we show that rooted tree (or forest) maps constructed in the previous section induce a class of relation among MZV's. 
This will be done by use of Kawashima relation, which we recall in the following.

\subsection{Kawashima relation}\label{subsec4.1}
Let $z_k:=x^{k-1}y$ for $k\ge 1$. The harmonic (or stuffle) product $\ast :\FH^1\times\FH^1\to\FH^1$ is a $\Q$-bilinear map defined by the following rules. 
\begin{align*}
\mathrm{i})& \quad\text{For any }w\in\FH^1,\ 1\ast w=w\ast 1=w. \\
\mathrm{ii})&  \quad\text{For any }w, w^{\prime}\in\FH^1\text{ and any }k,l\ge 1, \\
{}& \quad z_kw\ast z_lw^{\prime}=z_k(w\ast z_lw^{\prime})+z_l(z_kw\ast w^{\prime})+z_{k+l}(w\ast w^{\prime}). 
\end{align*}
This is, as shown in \cite{H}, an associative and commutative product on $\FH^1$. 

Denote by $\varphi$ an automorphism of $\FH$ defined by $\varphi(x)=z=x+y$ and $\varphi(y)=-y$. 
The linear part of Kawashima's relation \cite[Corollary 4.9]{Kaw} is then stated as follows. 

\begin{prop}\label{linK}
{\it $L_x\varphi(\FH y\ast\FH y)\subset \ker Z$.}
\end{prop}

Let $\tau$ be an anti-automorphism of $\FH$ defined by $\tau(x)=y$ and $\tau(y)=x$, which is known to induce the duality for MZV's: $(1-\tau)(\FH^0)\subset \mathrm{ker} Z$. In \cite{Kaw}, Kawashima proved that Kawashima's relation contains the duality formula:
\begin{lem}\label{lem4.1}
$(1-\tau)(\FH^0)\subset L_x\varphi(\FH y\ast\FH y). $
\end{lem}

\subsection{Main result 2}\label{subsec4.2}
For $w\in\FH^1$, let $\mathcal{H}_w(v):=w\ast v\ (v\in\FH^1)$. Denote by $\FH_n^1$ the degree $n$ homogenous part of $\FH^1$. 
Let $\mathfrak{W}$ be the $\Q$-vector space generated by $\{\mathcal{H}_w|w\in\FH^1\}$, and $\mathfrak{W}_n$ the vector subspace of $\mathfrak{W}$ generated by $\{\mathcal{H}_w|w\in\FH^1_n\}$. Let $\mathfrak{W}^{\prime}$ be the $\Q$-vector space generated by $\{L_{z_k}\mathcal{H}_w|k\ge 1,\ w\in\FH^1\}$, and $\mathfrak{W}^{\prime}_n$ the vector subspace of $\mathfrak{W}^{\prime}$ generated by $\{L_{z_k}\mathcal{H}_w|1\le k\le n,\,w\in\FH^1_{n-k}\}$. The $\Q$-linear map $ \lambda : \mathfrak{W}^{\prime}\to\mathfrak{W}$ is defined by 
\[
\lambda(L_{z_k}\mathcal{H}_w)=\mathcal{H}_{z_kw}. 
\]
Here, we show the well-definedness of the map $\lambda$. Assume that
\begin{equation}
\displaystyle \sum_{(z_k,w)}C_{(z_k,w)}L_{z_k}\mathcal{H}_w=0\ (\in\mathfrak{W}), \label{20}
\end{equation}
where the sum is over a finite number of pairs of words $(z_k,w)$. Applying $(\ref{20})$ to $1\in\FH$, we have
$$ \displaystyle \sum_{(z_k,w)}C_{(z_k,w)}z_kw=0. $$
Then, for each $z_k$, we have
$$ \displaystyle \sum_wC_{(z_k,w)}w=0 $$
where the sum is over different words $w$. Therefore, each coefficient $C_{(z_k,w)}$ becomes zero, and hence, $L_{z_k}\mathcal{H}_w$'s are linearly independent. 
We also set $\chi_x:=\tau L_x\varphi$. Then we have the following. 
\begin{thm}\label{thm4.4}
Let $n$ be a positive integer. For any rooted forest map $f$ with $\deg f=n$, we have
\begin{itemize}
\item[(A)] $\varphi\tau\phi_f R_x\tau\varphi \in\mathfrak{W}_n^{\prime}$. 
\item[(B)] $\chi_x^{-1}f\chi_x=-\lambda (\varphi\tau\phi_f R_x\tau\varphi)\in\mathfrak{W}_n$. 
\end{itemize}
\end{thm}
\begin{rem}
In (B), the expression $\chi_x^{-1}=\varphi\tau R_y^{-1}$ makes sense because (b) in the Introduction has been shown in the previous section. 
\end{rem}
\begin{proof}[Proof of Theorem \ref{thm4.4}]
We begin with the case of $n=1$. We have
\begin{equation}\label{eq26}
\varphi\tau\phi_{\includegraphics[width=2mm,height=1.8mm]{deg1.png}}R_x\tau\varphi =-L_y\in\mathfrak{W}_1^{\prime}, 
\end{equation}
and hence (A) holds. Because of (a) and (b) in the Introduction, we find
\begin{equation}\label{eq27}
R_y^{-1}\includegraphics[width=2.8mm,height=2.5mm]{deg1.png}R_y=R_y^{-1}(R_y\includegraphics[width=2.8mm,height=2.5mm]{deg1.png}-\psi_{\includegraphics[width=2mm,height=1.8mm]{deg1.png}})=\includegraphics[width=2.8mm,height=2.5mm]{deg1.png}-\phi_{\includegraphics[width=2mm,height=1.8mm]{deg1.png}}R_x. 
\end{equation}
We also calculate
\begin{align*}
\includegraphics[width=2.8mm,height=2.5mm]{deg1.png}\tau\varphi L_{z_k}&=-\includegraphics[width=2.8mm,height=2.5mm]{deg1.png}R_{z^{k-1}}R_x\tau\varphi \\
&=-R_{z^{k-1}}(\psi_{\includegraphics[width=2mm,height=1.8mm]{deg1.png}}+R_x\includegraphics[width=2.8mm,height=2.5mm]{deg1.png})\tau\varphi \\
&=R_{z^{k-1}}(R_y\phi_{\includegraphics[width=2mm,height=1.8mm]{deg1.png}}R_x+R_x\includegraphics[width=2.8mm,height=2.5mm]{deg1.png})\tau\varphi
\end{align*}
by using (a), (d) in the Introduction and Lemma \ref{lem4.3.0}. Hence we have
\begin{align*}
[\chi_x^{-1}\includegraphics[width=2.8mm,height=2.5mm]{deg1.png}\chi_x,L_{z_k}]&=\chi_x^{-1}\includegraphics[width=2.8mm,height=2.5mm]{deg1.png}\chi_x L_{z_k}-L_{z_k}\chi_x^{-1}\includegraphics[width=2.8mm,height=2.5mm]{deg1.png}\chi_x \\
&=\varphi\tau (\includegraphics[width=2.8mm,height=2.5mm]{deg1.png}-\phi_{\includegraphics[width=2mm,height=1.8mm]{deg1.png}}R_x)\tau\varphi L_{z_k}-L_{z_k}\varphi\tau (\includegraphics[width=2.8mm,height=2.5mm]{deg1.png}-\phi_{\includegraphics[width=2mm,height=1.8mm]{deg1.png}}R_x)\tau\varphi \\
&=-\varphi\tau R_{z^{k-1}}(R_y\phi_{\includegraphics[width=2mm,height=1.8mm]{deg1.png}}R_x+R_x \includegraphics[width=2.8mm,height=2.5mm]{deg1.png})\tau\varphi - \varphi\tau \phi_{\includegraphics[width=2mm,height=1.8mm]{deg1.png}}R_x\tau\varphi L_{z_k}-L_{z_k}\varphi\tau (\includegraphics[width=2.8mm,height=2.5mm]{deg1.png}-\phi_{\includegraphics[width=2mm,height=1.8mm]{deg1.png}}R_x)\tau\varphi \\
&=-L_{x^k}\varphi\tau\phi_{\includegraphics[width=2mm,height=1.8mm]{deg1.png}}R_x\tau\varphi -\varphi\tau\phi_{\includegraphics[width=2mm,height=1.8mm]{deg1.png}}R_x\tau\varphi L_{z_k} \\
&=[\lambda (-\varphi\tau\phi_{\includegraphics[width=2mm,height=1.8mm]{deg1.png}}R_x\tau\varphi ),L_{z_k}].
\end{align*}
Here we use Lemma \ref{lem4.3.1} and \eqref{eq26} for the last equality. By \eqref{eq27} and $\includegraphics[width=2.8mm,height=2.5mm]{deg1.png}(1)=0$, 
\[
\chi_x^{-1}\includegraphics[width=2.8mm,height=2.5mm]{deg1.png}\chi_x (1)=\varphi\tau (\includegraphics[width=2.8mm,height=2.5mm]{deg1.png}-\phi_{\includegraphics[width=2mm,height=1.8mm]{deg1.png}}R_x)\tau\varphi (1)=-\varphi\tau\phi_{\includegraphics[width=2mm,height=1.8mm]{deg1.png}}R_x\tau\varphi (1), 
\]
and by Lemma \ref{lem4.3.4} this is equal to $\lambda (-\varphi\tau\phi_{\includegraphics[width=2mm,height=1.8mm]{deg1.png}}R_x\tau\varphi )(1)$. Therefore we conclude (B) for $f=\includegraphics[width=2.8mm,height=2.5mm]{deg1.png}$ by using Lemma \ref{lem4.3.5}. 

Now suppose that (A) and (B) hold for any rooted forest map of degree $<n$ and let $f$ be any rooted forest map of degree $n$. We remark that
\begin{equation}\label{eq29}
R_y^{-1}f R_y=R_y^{-1}(R_y f-\psi_f)=f-\phi_f R_x,
\end{equation}
which is because of (a) and (b) in the Introduction. We obtain
\begin{equation}\label{eq28}
\varphi\tau f\tau\varphi =\chi_x^{-1}f\chi_x +\varphi\tau\phi_f R_x\tau\varphi =(\lambda -1)(-\varphi\tau\phi_fR_x\tau\varphi )\in (\lambda -1)(\mathfrak{W}_{n-1}^{\prime})
\end{equation}
because of \eqref{eq29} and (B). According to (e) in the Introduction, we have the expression
\[
\phi_f =\sum_{j=0}^n d_j R_{z^{n-j}}\quad (d_j\in\Q[\text{rooted tree maps}]_{(j)}), 
\]
and hence 
\[
\phi_f R_x\tau\varphi =\sum_{j=0}^n d_j R_{z^{n-j}}R_x\tau\varphi =\sum_{j=0}^n d_j\tau\varphi L_{z_{n+1-j}}. 
\]
We find 
\[
\varphi\tau d_j\tau\varphi\in (\lambda -1)(\mathfrak{W}_j^{\prime})\quad (1\leq j\leq n)
\]
because of \eqref{eq28} and Lemma \ref{lem4.3.3}. Therefore we obtain
\[
\varphi\tau\phi_f R_x\tau\varphi\in\Q\cdot L_{z_{n+1}}+\sum_{j=1}^n (\lambda -1)(\mathfrak{W}_j^{\prime})L_{z_{n+1-j}}\subset\mathfrak{W}_{n+1}^{\prime}, 
\]
which is expected as (A) for $f$. 

We calculate
\[
f\tau\varphi L_{z_{k}}=-R_{z^{k-1}}f R_x\tau\varphi =-R_{z^{k-1}}(R_y\phi_f R_x+R_x f) \tau\varphi
\]
by using (a), (d) in the Introduction and Lemma \ref{lem4.3.0}. Hence ,by using \eqref{eq29} and similar calculation above, we have
\[
[\chi_x^{-1}f\chi_x, L_{z_k}]=[\lambda (-\varphi\tau\phi_f R_x\tau\varphi), L_{z_k}]. 
\]
For this equality, we use Lemma \ref{lem4.3.1} and (A) for $f$ which has already been obtained. By \eqref{eq29} and $f(1)=0$, 
\[
\chi_x^{-1} f\chi_x (1)=\varphi\tau (f-\phi_f R_x)\tau\varphi (1)=-\varphi\tau\phi_f R_x\tau\varphi (1),
\]
which is found to be equal to $\lambda (-\varphi\tau\phi_f R_x\tau\varphi)(1)$ by using Lemma \ref{lem4.3.4}. Therefore we conclude (B) for $f$ by using Lemma \ref{lem4.3.5}. This completes the proof. 
\end{proof}
\begin{cor}
For any rooted forest map $f\neq\I$, there is an element $w\in\FH y$ such that
\[
f\chi_x =\chi_x\mathcal{H}_w. 
\]
\end{cor}
\begin{rem}
Such $w$ in the corollary is determined by
\[
w=\mathcal{H}_w(1)=\chi_x^{-1}f\chi_x(1)=\chi_x^{-1}f(y). 
\]
\end{rem}
\begin{cor}\label{cor4.8}
For any rooted forest map $f\neq\I$, we have
\[
f(\FH^0)\subset\ker Z.
\]
\end{cor}
\begin{proof}
It is enough to show, for any rooted forest map $f$, 
\[
f(x\FH y)\subset\ker Z
\]
because of $\FH^0=\Q+x\FH y$ and $f(\Q)=\{0\}$. 

By definition of $\varphi$ and $\tau$, we find
\[
\chi_x (\FH y)=x\FH y.
\]
By the previous corollary, there exists $w\in\FH y$ such that
\[
f\chi_x =\chi_x \mathcal{H}_w.
\]
We also notice that
\[
\chi_x(\FH y\ast\FH y)=(1-(1-\tau))(\FH y\ast\FH y)\subset L_x\varphi(\FH y\ast\FH y)
\]
due to Lemma \ref{lem4.1}. Therefore we have
\[
f(x\FH y)=f\chi_x (\FH y)=\chi_x\mathcal{H}_w (\FH y)\subset L_x\varphi(\FH y\ast\FH y). 
\]
Thanks to Proposition \ref{linK}, we have the conclusion. 
\end{proof}
\noindent As a consequence, we have Theorem \ref{mthm3}. 
\begin{rem}
Actually, rooted tree maps given by the natural growth $N^k(\I)$ of $\I$, is nothing but the quasi-derivation operator $(k-1)!\partial_k^{(1)}$ described in \cite{Kan,T}. 
\end{rem}

\subsection{Lemmata}\label{subsec4.3}
Following lemmata are required in the proof of main result 2 in the previous section.  

\begin{lem}\label{lem4.3.0}
$\varphi L_x=L_z\varphi,\ \varphi L_y=-L_y\varphi,\ \tau L_x=R_y\tau,\ \tau L_y=R_x \tau,\ \tau R_x=L_y\tau,\ \tau R_y=L_x\tau .$
\end{lem}
\begin{proof}
Easy. 
\end{proof}

\begin{lem}\label{lem4.3.1}
{\it For any $X\in\mathfrak{W}^{\prime}$ and any $l\ge 1$, we have $[\lambda(X),L_{z_l}]=XL_{z_l}+L_{x^l}X$. }
\end{lem}
\begin{proof}
It is sufficient to show the case in which $X=L_{z_k}\mathcal{H}_w$, which follows directly from 
\begin{equation}
[\mathcal{H}_{z_kw},L_{z_l}]=L_{z_k}\mathcal{H}_wL_{z_l}+L_{z_{k+l}}\mathcal{H}_w, \label{16}
\end{equation}
the harmonic product rule.
\end{proof}

\begin{lem}\label{lem4.3.2}
{\it For any $k,l\ge 1$, we have $(\lambda -1)(\mathfrak{W}^{\prime}_k)L_{z_l}\subset \mathfrak{W}^{\prime}_{k+l}$. }
\end{lem}
\begin{proof}
The proof follows directly from (\ref{16}). 
\end{proof}

\begin{lem}\label{lem4.3.3}
{\it We have $(\lambda -1)(\mathfrak{W}^{\prime}_k)\cdot(\lambda -1)(\mathfrak{W}^{\prime}_l)\subset(\lambda -1)(\mathfrak{W}^{\prime}_{k+l})$ for any $k,l\ge 1$. }
\end{lem}
\begin{proof} 
Let $d$ and $d^{\prime}$ be the weights of words $w$ and $w^{\prime}$, respectively. The assertion $(\lambda -1)(L_{z_k}\mathcal{H}_w)\cdot(\lambda -1)(L_{z_l}\mathcal{H}_{w^{\prime}})\in (\lambda -1)(\mathfrak{W}^{\prime}_{k+l+d+d^{\prime}})$ is only necessary to show.
\begin{eqnarray*}
\text{LHS} &=& (\mathcal{H}_{z_kw}-L_{z_k}\mathcal{H}_w)(\mathcal{H}_{z_lw^{\prime}}-L_{z_l}\mathcal{H}_{w^{\prime}}) \\
&=& \mathcal{H}_{z_kw\ast z_lw^{\prime}}-\mathcal{H}_{z_kw}L_{z_l}\mathcal{H}_{w^{\prime}}-L_{z_k}\mathcal{H}_{w\ast z_lw^{\prime}}+L_{z_k}\mathcal{H}_wL_{z_l}\mathcal{H}_{w^{\prime}} \\
&=& \mathcal{H}_{z_k(w\ast z_lw^{\prime})+z_l(z_kw\ast w^{\prime})+z_{k+l}(w\ast w^{\prime})}-(L_{z_k}\mathcal{H}_wL_{z_l} \\
&{}& \quad +L_{z_l}\mathcal{H}_{z_kw}+L_{z_{k+l}}\mathcal{H}_w)\mathcal{H}_{w^{\prime}}-L_{z_k}\mathcal{H}_{w\ast z_lw^{\prime}}+L_{z_k}\mathcal{H}_wL_{z_l}\mathcal{H}_{w^{\prime}} \\
&=& \mathcal{H}_{z_k(w\ast z_lw^{\prime})}-L_{z_k}\mathcal{H}_{w\ast z_lw^{\prime}}+\mathcal{H}_{z_l(z_kw\ast w^{\prime})}-L_{z_l}\mathcal{H}_{z_kw\ast w^{\prime}} \\
&{}& \quad +\mathcal{H}_{z_{k+l}(w\ast w^{\prime})}-L_{z_{k+l}}\mathcal{H}_{w\ast w^{\prime}} \\
&=& (\lambda -1)(L_{z_k}\mathcal{H}_{w\ast z_lw^{\prime}}+L_{z_l}\mathcal{H}_{z_kw\ast w^{\prime}}+L_{z_{k+l}}\mathcal{H}_{w\ast w^{\prime}}). \\
&\in& \mathrm{RHS}.
\end{eqnarray*}
Hence, the lemma is proven. 
\end{proof}

\begin{lem}\label{lem4.3.4}
{\it For any $X\in\mathfrak{W}^{\prime}$, we have $\lambda(X)(1)=X(1)$.}
\end{lem}
\begin{proof}
$(\lambda-1)(L_{z_k}\mathcal{H}_w)(1)=\mathcal{H}_{z_kw}(1)-L_{z_k}\mathcal{H}_w(1)=z_kw-z_kw=0.$
\end{proof}

\begin{lem}\label{lem4.3.5}
{\it Let $X\in\mathfrak{W}$. If $X(1)=0$ and $[X,L_{z_k}]=0$ for any $k\ge 1$, we have $X=0$.}
\end{lem}
\begin{proof}
If $[X,L_{z_k}]=0$ for any $k\ge 1$,
$$ X(z_{k_1}\cdots z_{k_n})=z_{k_1}X(z_{k_2}\cdots z_{k_n}) = \cdots = z_{k_1}\cdots z_{k_n}X(1)=0. $$
\end{proof}

\end{document}